\documentclass[12pt, reqno, a4paper]{amsart}
\usepackage{times}

\usepackage[utf8]{inputenc}
\usepackage[USenglish]{babel}
\usepackage{amsmath,amsthm,amssymb,amsfonts}
\usepackage{mathtools}
\usepackage{mathrsfs}
\usepackage{bbm}
\usepackage{booktabs}

\usepackage{indentfirst}
\usepackage{enumitem}
\usepackage{xcolor}
\usepackage{graphicx}

\usepackage{float}

\usepackage[breaklinks=true, bookmarksopenlevel=1, bookmarksdepth=2]{hyperref}
\hypersetup{
colorlinks,
   linkcolor={cyan!80!black},
   citecolor={cyan!80!black},
 urlcolor={cyan!80!black}
}

\allowdisplaybreaks

\newtheorem{thm}{Theorem}[section]
\newtheorem{cor}[thm]{Corollary}
\newtheorem{lem}[thm]{Lemma}
\newtheorem{prop}[thm]{Proposition}

\theoremstyle{plain} % just in case the style had changed
\newcommand{\thistheoremname}{}
\newtheorem*{genericthm}{\thistheoremname}

\theoremstyle{definition}

\newtheorem*{prob}{Problem}

\theoremstyle{remark}

\newtheorem*{ntt}{Notation}

\numberwithin{equation}{section}

\newcommand{\N}{\mathbb{N}}      % N = Naturals
\newcommand{\Z}{\mathbb{Z}}      % Z = Integers
\newcommand{\Q}{\mathbb{Q}}      % Q = Rationals
\newcommand{\R}{\mathbb{R}}      % R = Reals
\newcommand{\eps}{\varepsilon}   % epsilon

\newcommand{\ul}[1]{\mathbf{#1}}  % Underlined letter (vector)

\newcommand{\logmap}{\operatorname{logmap}} 
\newcommand{\vol}{\operatorname{vol}}
\newcommand{\NVol}{\widehat{\operatorname{vol}}}
\newcommand{\conv}{\operatorname{Conv}}
\renewcommand{\P}{\mathbb{P}}

\usepackage[margin=1in]{geometry}

\restylefloat{table}

\setlist[itemize]{leftmargin=*}
\setlist[enumerate]{leftmargin=*}

\begin{document}

\title{The multiplication table problem in large dimensions}%
\author{Cihan Sabuncu}
\address{Max Planck Institute for Mathematics\\
Vivatsgasse 7\\
53111 Bonn\\
Germany}
\curraddr{}
\email{sabuncu@mpim-bonn.mpg.de}
\thanks{}

\author{Christian T\'afula}
\address{Institute of Mathematics, Statistics\\
and Computer Science\\
University of S\~ao Paulo\\
Rua do Mat\~ao, 1010\\
S\~ao Paulo, SP 05508-090\\
Brazil}
\curraddr{}
\email{tafula@ime.usp.br}
\thanks{C. T\'afula was supported by the São Paulo Research Foundation (FAPESP), Brazil, Process No.~2025/15961-3.}


% ----------------------------------------------------------------
 \begin{abstract}
  For $N\geq 2$ and $k\geq 1$, let $M_k(N):=\#\{x_1\cdots x_k : x_i\in\{1,\ldots,N\}\text{ for all } i\}$ be the $k$-dimensional multiplication table. Given $N$, Khovanskii's theorem implies that $M_k(N)$ agrees, for all sufficiently large $k$, with a polynomial in $k$ of degree $\pi(N)$. We determine the asymptotic size of its leading coefficient, proving that, as $N\to\infty$, with $k$ sufficiently large relative to $N$,
  \[ M_k(N) = \exp\bigg((2\pi+o(1))\frac{\sqrt{N}}{\log N}\bigg)\frac{k^{\pi(N)}}{\pi(N)!}. \]
  We also study the analogous problem when the factors are restricted to $y$-smooth integers. For $y=o(\log N)$, we prove that the number of distinct products of $k$ such integers up to $N$ is asymptotic to the number of $y$-smooth integers up to $N^k$, uniformly for $k\geq 1$.
 \end{abstract}

\maketitle
% ----------------------------------------------------------------

%%%%%
\section{Introduction}
 The multiplication table problem, first posed by Erd\H{o}s in 1955, asks how many distinct entries occur in the $N\times N$ multiplication table. Writing $[N]:=\{1,\ldots,N\}$, let $M_2(N):=\#\{x_1x_2 : x_1,x_2\in[N]\}$. Ford \cite{for08} showed in 2008 that
 \[ M_2(N)\asymp \frac{N^2}{(\log N)^{\delta}(\log\log N)^{3/2}}, \]
 where $\delta:=1-\frac{1+\log\log 2}{\log 2} = 0.086\ldots$ is the multiplication table constant. Recently, Green and Sawhney \cite{GreenSawhneyMultiplicationTable,green2026proportionpermutationsfixingkset} announced an asymptotic formula for this quantity, with an oscillatory function appearing in the leading term.

 The problem extends naturally to higher dimensions. For $k\geq 2$, let $M_k(N):=\#\{x_1\cdots x_k : x_i\in[N]\text{ for all }i\}$. For fixed $k$, Koukoulopoulos \cite{kou10} showed in 2010 that, writing $Q(u):=\int_1^u\log t\,\mathrm{d}t=u\log u-u+1$, we have
 \[ M_k(N)\asymp_k \frac{N^k}{(\log N)^{Q((k-1)/\log k)}(\log\log N)^{3/2}}. \]
 These results concern the case where $k$ is fixed while $N\to\infty$. A natural complementary question is to understand what happens when the dimension $k$ is allowed to vary with $N$.

\subsection{Khovanskii's theorem and large dimensions}
 A natural framework for this problem comes from Khovanskii's theorem on iterated sumsets. Let $d\geq 1$ be an integer and let $A\subseteq\Z^d$ be a finite set. For each $k\geq 1$, define $kA := \{\ul{a}_1+\cdots+\ul{a}_k : \ul{a}_i\in A \text{ for all } i\}$. Khovanskii \cite{kho92} showed that, for sufficiently large $k$, the cardinality $|kA|$ is given by a polynomial $\mathcal{P}_A(x)\in\Q[x]$ of degree at most $d$; that is, $|kA|=\mathcal{P}_A(k)$ for all sufficiently large $k$. The coefficients of $\mathcal{P}_A(x)$ encode geometric and combinatorial properties of $A$; in particular, when $A-A$ generates $\Z^d$, its leading term is determined by the volume of the convex hull of $A$ \cite{curgol21,grashawal23,grasmiwal24}. This has also been generalized to abelian semigroups by Nathanson--Ruzsa \cite{natruz02}.

 To apply this to the multiplication table problem, consider the map
 \[ \logmap : [N] \longrightarrow \Z^{\pi(N)},\qquad n \longmapsto (v_p(n))_{p\leq N}. \]
 By uniqueness of prime factorization, this map is injective, and it transforms multiplication into addition. In particular, $M_k(N) = |k\logmap([N])|$. Thus the multiplication table problem in large dimensions becomes a problem about iterated sumsets in $\Z^{\pi(N)}$. By Khovanskii's theorem, there exists $k_N\in\N$ such that $M_k(N) = \mathcal{P}_N(k)$ for every $k\geq k_N$, where
 \[ \mathcal{P}_N(k) = c_N k^{\pi(N)} + O_N(k^{\pi(N)-1}). \]
 Since $\logmap([N])$ contains $\ul{0}$ and $\ul{e}_p$ for every prime $p\leq N$, its convex hull has dimension $\pi(N)$, and the leading coefficient is given by $c_N = \vol(\conv(\logmap([N])))$. This connection between Khovanskii's theorem and the multiplication table problem was first observed by Limbach--Scheidweiler--Triesch \cite{limbach2025effectivekhovanskiiehrhartpolytopes}.
 
 Our main goal will be to determine the asymptotic size of $c_N$ as $N\to\infty$.
 
\subsection{A smooth variant}
 We start by studying a variant in which the entries in the multiplication table are restricted to smooth numbers. Let $P^{+}(n)$ denote the largest prime factor of $n$, and define
 \[ M_k(N,y) := \#\{x_1\cdots x_k : x_i\in[N],\ P^{+}(x_i)\leq y\text{ for all }i\}. \]
 As usual, write $\Psi(x,y):=\#\{n\leq x : P^{+}(n)\leq y\}$. See \cite{MR2467549} for more details on these numbers. 
 
 For fixed $N$ and $2\leq y\leq N$, Khovanskii's theorem gives
 \[ M_k(N,y)=c_{N,y}k^{\pi(y)}+O_{N,y}(k^{\pi(y)-1}) \]
 for all sufficiently large $k$, where $c_{N,y}$ is the volume of the convex hull of the exponent vectors $(v_p(n))_{p\leq y}$ of the $y$-smooth integers up to $N$. In Section \ref{smooth_section}, we first show that, throughout the range $y=o(\log N)$,
 \[ c_{N,y} \sim \frac{1}{\pi(y)!}\prod_{p\leq y}\frac{\log N}{\log p}. \]
 Thus, in terms of volume, the convex hull is asymptotically as large as the ambient weighted simplex
 \[ \Sigma(N,y) := \bigg\{\ul{u}=(u_p)_{p\leq y}\in\R_{\geq 0}^{\pi(y)} : \sum_{p\leq y} u_p\log p\leq\log N\bigg\}. \]

 In the smaller range $2\leq y\leq\sqrt{\log N\log\log N}$, a result of Ennola \cite[Theorem 1]{MR244175} gives
 \[ \Psi(N,y) \sim \frac{1}{\pi(y)!} \prod_{p\leq y} \frac{\log N}{\log p}, \]
 and hence $c_{N,y} \sim \Psi(N,y)$. It follows that, for $k$ sufficiently large in terms of $N$ and $y$,
 \[ M_k(N,y)\sim k^{\pi(y)}\Psi(N,y) \sim \Psi(N^k,y). \]
 On the other hand, Mehdizadeh \cite[Theorem 1.2]{MR4177515} proved that
 \[ M_2(N,y) \sim \Psi(N^2,y) \]
 in the larger range $y\leq \exp\left(\frac{(\log N)^{1/3}}{(\log\log N)^{1/3+\varepsilon}}\right)$ for any small $\varepsilon > 0$. Our first result shows that this asymptotic in fact holds uniformly in the dimension, throughout the range $y=o(\log N)$.

 \begin{thm}\label{smooth_thm}
  Let $y=y(N)$ satisfy $y=o(\log N)$. Then, uniformly for $k\geq 1$,
  \[ M_k(N,y) \sim \Psi(N^k,y). \]
 \end{thm}
 
 Theorem \ref{smooth_thm} naturally raises the question of how far the range of the smoothness parameter can be extended.

 \begin{prob}
  How large can $y=y(N)$ be while
  \[ M_k(N,y)\sim \Psi(N^k,y) \]
  continues to hold uniformly for $k\geq 1$?
 \end{prob}

\subsection{Estimating \texorpdfstring{$c_N$}{c_N}}
 To obtain some initial bounds, consider the simplex
 \[ \Sigma_N:=\bigg\{\ul{u}=(u_p)_{p\leq N}\in\R_{\geq 0}^{\pi(N)}:\sum_{p\leq N}u_p\log p\leq\log N\bigg\}. \]
 By definition, $\logmap([N])=\Sigma_N\cap\Z_{\geq 0}^{\pi(N)}$. Hence
 \[ \conv\bigg(\ul{0},\bigg\{\bigg\lfloor\frac{\log N}{\log p}\bigg\rfloor \ul{e}_p : p\leq N\bigg\}\bigg) \subseteq \conv(\logmap([N])) \subseteq \Sigma_N, \]
 since $\lfloor\frac{\log N}{\log p}\rfloor \ul{e}_p \in \logmap([N])$ for every prime $p\leq N$. The volumes of the two outer simplices are respectively $\frac{1}{\pi(N)!}\prod_{p\leq N}\lfloor\frac{\log N}{\log p}\rfloor$ and $\frac{1}{\pi(N)!}\prod_{p\leq N}\frac{\log N}{\log p}$. Writing $V_N := \pi(N)!c_N$, we obtain
 \[ \prod_{p\leq N}\bigg\lfloor\frac{\log N}{\log p}\bigg\rfloor \leq V_N \leq \prod_{p\leq N}\frac{\log N}{\log p}. \]
 Taking logarithms, these bounds give
 \[ 2\log 2\,\frac{\sqrt{N}}{\log N} \lesssim \log V_N \lesssim \frac{N}{(\log N)^2}. \]
 Our main result shows that the correct order of magnitude is that of the lower bound, and determines the precise constant.

 \begin{thm}\label{thm2pi}
  We have
  \[ \log V_N\sim 2\pi\frac{\sqrt{N}}{\log N}. \]
 \end{thm}
 
 Consequently, by Khovanskii's theorem, for $k$ sufficiently large in terms of $N$,
 \[ M_k(N) = \exp\bigg((2\pi+o_{N\to\infty}(1))\frac{\sqrt{N}}{\log N}\bigg)\frac{k^{\pi(N)}}{\pi(N)!}. \]

 The appearance of $\sqrt{N}$ is suggested by the arithmetic structure of $\logmap([N])$. For $x\leq N$, write $P_x:=\conv(\logmap([x]))$, and let $P(N,y)$ denote the convex hull of the exponent vectors $(v_p(n))_{p\leq N}$ of the $y$-smooth integers up to $N$. Every integer $n\leq N$ either has all its prime factors at most $\sqrt{N}$, or has a unique prime factor $p>\sqrt{N}$, in which case $n=pm$ with $m\leq N/p$. Thus, the proof begins with the decomposition
 \[ P_N=\conv\bigg(P(N,\sqrt{N})\cup\bigcup_{\sqrt{N}<p\leq N}\big(\ul{e}_p+P_{N/p}\big)\bigg). \]
 After a diagonal change of variables, the pieces occurring in this decomposition can be compared with standard simplices. Grouping the primes $p>\sqrt{N}$ into geometric ranges then produces upper and lower model bodies whose logarithmic volumes can be analyzed explicitly. The resulting volume computation reduces, in the limit, to a Riemann sum, and the constant in Theorem \ref{thm2pi} arises from the integral
 \[ 2\int_0^1 \big((1+x^{-2})\log(1+x^{-2}) - x^{-2}\log(x^{-2})\big)\,\mathrm{d}x = 2\pi. \]

 It is also natural to ask how large $k$ must be before the polynomial behavior predicted by Khovanskii's theorem begins. Bounds for $k_N$ can be obtained from general results on the size and stability of sumsets in $\Z^d$. A result of Granville, Smith, and Walker \cite[Theorem 1.5]{grasmiwal24} gives
 \[ k_N\leq N^2V_N-N+1. \]
 %We believe this bound to be somewhat optimal in the general setting.

 \begin{ntt}
  We write $v_p(n)$ for the $p$-adic valuation of $n$, $P^{+}(n)$ for the largest prime factor of $n$, with the convention $P^{+}(1):=1$, and $\pi(x)$ for the number of primes up to $x$. For a set $A\subseteq\R^d$, we denote by $\conv(A)$ its convex hull and by $\vol(A)$ its Euclidean volume. If $A$ is a $d$-dimensional convex body, we write $\NVol(A):=d!\vol(A)$ for its normalized volume. We write $\ul{e}_i$ for the standard basis vectors, or $\ul{e}_p$ when the coordinates are indexed by primes.
 \end{ntt}

%%%%%%%%%%%%%%%%%%%%%%%%%%%%%%%%%%%%%%%%%%%%%%%%%%%%%%%%%%%%%%%%%%
\section{A smooth variant}\label{smooth_section}
 We begin by estimating the leading coefficient $c_{N,y}$ introduced in the previous section.

 \begin{lem}\label{smooth_leading_coeff}
  Let $y=y(N)$ satisfy $y=o(\log N)$. Then
  \[ c_{N,y}\sim \frac{1}{\pi(y)!}\prod_{p\leq y}\frac{\log N}{\log p}. \]
 \end{lem}
 \begin{proof}
  By the trivial bounds,
  \[ \frac{1}{\pi(y)!}\prod_{p\leq y}\bigg\lfloor\frac{\log N}{\log p}\bigg\rfloor\leq c_{N,y}\leq\frac{1}{\pi(y)!}\prod_{p\leq y}\frac{\log N}{\log p}. \]
  Hence it is enough to show that
  \[ \prod_{p\leq y}\bigg\lfloor\frac{\log N}{\log p}\bigg\rfloor\sim\prod_{p\leq y}\frac{\log N}{\log p}. \]
  Uniformly for $p\leq y$, we have $\frac{\log N}{\log p}\to\infty$, since $y=o(\log N)$, and therefore
  \[ \bigg\lfloor\frac{\log N}{\log p}\bigg\rfloor=\frac{\log N}{\log p}\bigg(1+O\bigg(\frac{\log p}{\log N}\bigg)\bigg). \]
  Multiplying over $p\leq y$, we obtain
  \[ \prod_{p\leq y}\bigg\lfloor\frac{\log N}{\log p}\bigg\rfloor=\bigg(\prod_{p\leq y}\frac{\log N}{\log p}\bigg)\exp\bigg(O\bigg(\frac{1}{\log N}\sum_{p\leq y}\log p\bigg)\bigg). \]
  Since $\sum_{p\leq y}\log p\ll y=o(\log N)$, the exponential factor is $\exp(o(1))=1+o(1)$.
 \end{proof}

 We now prove Theorem \ref{smooth_thm}.
 
 \begin{proof}[Proof of Theorem \ref{smooth_thm}]
  We first prove
  \begin{equation}\label{smooth_upper_lower_bound}
    \Psi\bigg(\frac{N^k}{y^{k-1}},y\bigg)
    \leq M_k(N,y)
    \leq \Psi(N^k,y).
  \end{equation}
  The upper bound is immediate. For the lower bound, let $n\leq N^k/y^{k-1}$ be $y$-smooth. If $n\leq N$, then we are done. Otherwise, write $n=p_1\cdots p_\ell$ with $p_1\leq p_2\leq\cdots\leq p_\ell\leq y$. We can find $1\leq s<\ell$ such that
  \[ \prod_{i\leq s}p_i \leq N  < \prod_{i\leq s+1}p_i \leq y\prod_{i\leq s}p_i. \]
  Hence, if we let $d:=\prod_{i\leq s}p_i$, then $N/y<d\leq N$, and therefore
  \[ \frac{n}{d} \leq \frac{N^{k-1}}{y^{k-2}}. \]
  Proceeding in the same way with $n/d$, and iterating this argument, we obtain a factorization of $n$ into $k$ integers, each at most $N$. This proves \eqref{smooth_upper_lower_bound}.

%%%%%%%%%
  It remains to prove that, for $y=o(\log N)$,
  \[ \Psi\bigg(\frac{N^k}{y^{k-1}},y\bigg) \sim \Psi(N^k,y), \]
  uniformly in $k$. Suppose first that $y\leq \sqrt{\log N}$. By Ennola's result \cite[Theorem 1]{MR244175},
  \begin{align*}
   \Psi\bigg(\frac{N^k}{y^{k-1}},y\bigg) &= \bigg(1+O\bigg(\frac{y^2}{\log y\log(N^k/y^{k-1})}\bigg)\bigg) \frac{1}{\pi(y)!} \prod_{p\leq y} \frac{\log(N^k/y^{k-1})}{\log p} \\
   &= \bigg(1+O\bigg(\frac{y^2}{\log y\log N^k}\bigg)\bigg) \bigg(1-\frac{(k-1)\log y}{k\log N}\bigg)^{\pi(y)} \frac{1}{\pi(y)!} \prod_{p\leq y}\frac{\log N^k}{\log p} \\
   &= \bigg(1+O\bigg(\frac{y^2}{\log y\log N^k}\bigg)\bigg) \bigg(1+O\bigg(\frac{\pi(y)\log y}{\log N}\bigg)\bigg) \Psi(N^k,y).
  \end{align*}
  Both error terms are $o(1)$ uniformly in $k$, giving the desired asymptotic.
  
  Now suppose that $\sqrt{\log N}<y=o(\log N)$. Applying \cite[Théorème 2.4(ii), case $m=1$]{MR2166385} with $x=N^k$ and $d=y^{k-1}$, we obtain
  \[ \frac{\Psi(N^k/y^{k-1},y)}{\Psi(N^k,y)} = \bigg(1+O\bigg(\frac{\log y}{y}+\frac{\log y}{\log N}\bigg)\bigg)\bigg(1-\frac{(k-1)^2}{u^2+\overline{u}^{\,2}}\bigg)^{b\overline{u}}\frac{1}{y^{(k-1)\alpha}}, \]
  where $u=k\log N/\log y$, $\overline{u}=\min\{u,y/\log y\}=y/\log y$, $\alpha=\alpha(N^k,y)$, and $b$ is a quantity bounded above and below by positive absolute constants. Indeed, in the notation of \cite{MR2166385}, $t=k-1$ and $u_y=\overline{u}+\log y/\log(u+2)$, so $1/u_y+t/u\ll \log y/y+\log y/\log N$. Moreover,
  \[ \overline{u}\frac{(k-1)^2}{u^2+\overline{u}^{\,2}}\leq \frac{y\log y}{(\log N)^2}=o(1), \]
  and hence the middle factor is $1+o(1)$ uniformly in $k$. On the other hand, by \cite[Lemme 3.1]{MR2166385},
  \[ \alpha=\bigg(1+O\bigg(\frac{1}{\log y}\bigg)\bigg)\frac{\log(1+\frac{y}{k\log N})}{\log y}, \]
  so $y^{-(k-1)\alpha}=1+o(1)$ uniformly in $k$. Therefore,
  \[ \Psi\bigg(\frac{N^k}{y^{k-1}},y\bigg)\sim\Psi(N^k,y) \]
  in this range as well. Together with \eqref{smooth_upper_lower_bound}, this completes the proof.
 \end{proof}

%%%%%%%%%%%%%%%%%%%%%%%%%%%%%%%%%%%%%%%%%%%%%%%%%%%%%%%%%%%%%%%%%%
\section{Estimating \texorpdfstring{$c_N$}{c\_N}}
 For each $N\geq 2$, we work in the space $\R^{\pi(N)}$, whose coordinates are indexed by the primes $p\leq N$. We write $\ul{e}_p\in \R^{\pi(N)}$ for the standard basis vector in the $p$-coordinate. If $x\leq N$, we identify $\R^{\pi(x)}$ with the coordinate subspace of $\R^{\pi(N)}$ spanned by $\{\ul{e}_p:p\le x\}$.

 For $n\le N$, we recall
 \[ \logmap(n) = (v_p(n))_{p\leq N}\in \R^{\pi(N)}. \]
 For $x\leq N$, define
 \[ P_x := \conv(\logmap([x]))\subseteq \R^{\pi(N)}, \qquad \Delta(x) := \conv(\ul{0},\{\ul{e}_p : p\le x\})\subseteq \R^{\pi(N)}, \]
 so that $V_N  = \pi(N)!\vol(P_N)$. If we denote by
 \[ \NVol(K) := d! \vol(K), \]
 the normalized volume for a $d$-dimensional body $K$, then we have $V_N = \NVol(P_N)$.
 
 \subsection{Upper and lower bounds}
 Define also, for $y\leq N$,
 \[ \mathcal{S}(N,y) := \{n\leq N : P^{+}(n) \leq y\},\qquad P(N,y) := \conv(\logmap(\mathcal{S}(N,y))) \subseteq \R^{\pi(N)}. \]
 
 \begin{lem}\label{PNdecomp}
  For $N\geq 2$, we have
  \[ P_N = \conv\bigg(P(N,\sqrt{N})\cup \bigcup_{\sqrt{N}<p\leq N} (\ul{e}_p + P_{N/p}) \bigg). \]
 \end{lem}
 \begin{proof}
  Let $n\leq N$. If every $P^{+}(n) \leq \sqrt{N}$, then $\logmap(n)\in P(N,\sqrt{N})$ by definition. Otherwise let $p := P^{+}(n)$. Since $p>\sqrt{N}$, we have $p^2>N$, hence $v_p(n)=1$. Therefore $n=pm$ for some integer $m\leq N/p$, and
  \[ \logmap(n)=\ul{e}_p+\logmap(m)\in \ul{e}_p+P_{N/p}. \]
  This shows that every vertex $\logmap(n)$ of $P_N$ lies in the right-hand side, so
  \[ P_N\subseteq \conv\bigg(P(N,\sqrt{N})\cup \bigcup_{\sqrt{N}<p\leq N} (\ul{e}_p + P_{N/p}) \bigg). \]

  Conversely, if $\sqrt{N}<p\leq N$ and $m\leq N/p$, then $pm\leq N$, so
  \[ \ul{e}_p+\logmap(m)=\logmap(pm)\in P_N. \]
  Thus every set $\ul{e}_p+P_{N/p}$ is contained in $P_N$, and clearly $P(N,\sqrt{N})\subseteq P_N$. Hence the reverse inclusion follows.
 \end{proof}

 Fix $0<\eps<1$ and $q\in(0,1)$, and let $k$ be the unique integer such that
 \[ q^k\geq \eps > q^{k+1}. \]
 For $1\leq i\leq k$, define
 \[ \mathcal{P}_i:= \bigg\{p\in \P:\frac{\sqrt{N}}{q^{i-1}}<p\leq \frac{\sqrt{N}}{q^i}\bigg\} \] %r_i = P_i
 and 
 \[ m_i^{-} := \lfloor q^i \sqrt{N}\rfloor, \qquad m_i^{+}:=\lfloor q^{i-1}\sqrt{N}\rfloor. \]
 Define
 \begin{align}
  L_N = L_N(\eps, q) &:= \conv\bigg(2\Delta(\sqrt{N})\cup \bigcup_{i=1}^k \bigcup_{p\in \mathcal{P}_i} (\ul{e}_p + \Delta(m_i^{-}))\bigg), \label{Lbody} \\
  U_N = U_N(\eps, q) &:= \conv\bigg(2\Delta(\sqrt{N})\cup \bigcup_{i=1}^k \bigcup_{p\in \mathcal{P}_i} (\ul{e}_p + \Delta(m_i^{+}))\bigg). \label{Ubody}
 \end{align}

 \begin{prop}\label{ublbPN}
  Let $N\geq 2$. For every $0<\eps<1$ and $q\in(0,1)$, we have
  \[ \log\NVol(L_N) \leq \log V_N \leq \log\NVol(U_N) + O\bigg(\eps\log\bigg(\frac{1}{\eps}\bigg)\frac{\sqrt{N}}{\log N}\bigg) + O\bigg(\frac{\sqrt{N}}{(\log N)^2}\bigg). \]
 \end{prop}
 
 To prove Proposition \ref{ublbPN}, we will need the following lemmas.
 
 %%%%%%%%%%%%%%%%%%%%%%%%%%%%%%%%%%%%%%%%%%%%%%%%%%%%%%%%%%%%%%%%%%
 \begin{lem}\label{outsimp}
  For $m, y \leq N$, recall
  \[ \Sigma(m,y) = \bigg\{\ul{u}=(u_p)_{p\leq N}\in \R_{\geq 0}^{\pi(N)} : \sum_{p\leq y} u_p\log p\le \log m,\ u_p = 0 \text{ for }p>y\bigg\}. \]
  Then $P(m,y)\subseteq \Sigma(m,y)$. Moreover, if $T:\R^{\pi(N)}\to \R^{\pi(N)}$ is the diagonal map
  \[ T(u_p):=\frac{\log p}{\log y}\,u_p \qquad (p\leq y) \]
  that acts as the identity on the coordinates $p>y$, then
  \[ T(\Sigma(m,y)) = \frac{\log m}{\log y}\,\Delta(y). \]
 \end{lem}
 \begin{proof}
  Since every $y$-smooth integer $n\leq m$ satisfies $\sum_{p\leq y} v_p(n)\log p = \log n\leq \log m$, we have $P(m,y)\subseteq \Sigma(m,y)$. Now for the equality, if $\ul{u}\in \Sigma(m,y)$, then
  \[ \sum_{p\leq y} T(u_p) = \frac{1}{\log y} \sum_{p\leq y} u_p\log p\leq \frac{\log m}{\log y}, \]
  so $T(\ul{u})\in \frac{\log m}{\log y}\Delta(y)$. Conversely, let $\ul{v} \in \frac{\log m}{\log y}\Delta(y)$. Then $\sum_{p\leq y} v_p\leq \frac{\log m}{\log y}$ and $v_p = 0$ for $p>y$. Define
  \[ u_p := \begin{cases}
             \dfrac{\log y}{\log p}\, v_p, &p\leq y, \\
             0, &p>y.
            \end{cases} \]
  Then
  \[ \sum_{p\leq y} u_p\log p = \log y \sum_{p\leq y} v_p \leq \log y\frac{\log m}{\log y} = \log m, \]
  so $\ul{u} \in \Sigma(m,y)$. Since $\ul{v} = T(\ul{u})$, this proves the reverse inclusion.
 \end{proof} 
 
 For the next two lemmas, we write $\ul{e}_1,\ldots, \ul{e}_n$ for the standard basis of $\R^n$.
 
 %%%%%%%%%%%%%%%%%%%%%%%%%%%%%%%%%%%%%%%%%%%%%%%%%%%%%%%%%%%%%%%%%%
 \begin{lem}\label{Nvol-comp}
  Let $K\subseteq \R^d\times \{0\}^s\subseteq \R^{d+s}$ be a $d$-dimensional convex body. Then
  \[ \NVol(\conv(K,\ul{e}_{d+1},\ldots,\ul{e}_{d+s})) = \NVol(K). \]
 \end{lem}
 \begin{proof}
  It is enough to prove that $\NVol(\conv(K,\ul{e}_{d+1})\bigr)=\NVol(K)$, since the general case then follows by iteration. The body $\conv(K,\ul{e}_{d+1})$ is a pyramid of height $1$ over the base $K$, hence
  \[ \vol_{d+1}(\conv(K,\ul{e}_{d+1})) = \frac{1}{d+1}\vol_d(K). \]
  Multiplying by $(d+1)!$, we get
  \[ (d+1)! \vol_{d+1}(\conv(K,\ul{e}_{d+1})) = d! \vol_d(K), \]
  which is exactly $\NVol(\conv(K,\ul{e}_{d+1})) =\NVol(K)$.
 \end{proof}
 
 %%%%%%%%%%%%%%%%%%%%%%%%%%%%%%%%%%%%%%%%%%%%%%%%%%%%%%%%%%%%%%%%%%
 \begin{lem}\label{mixedvol}
  Let $n\geq 1$, and let $Q\subseteq \R^n$ be a $q$-dimensional convex body for some $q\leq n$. Let $1\leq d\leq n$, and assume
  \[ K_1,\ldots,K_r\subseteq Q\cap (\R^d\times \{0\}^{n-d}) \]
  are convex bodies. Define
  \[ C:=\conv((Q,\ul{0}), (K_1,\ul{e}_{n+1}),\ldots,(K_r,\ul{e}_{n+r}))\subseteq \R^{n+r}. \]
  Then
  \[ \NVol(C)\leq \binom{d+r}{r}\NVol(Q). \]
 \end{lem}
 \begin{proof}
  Let $\pi:\R^{n+r}\to \R^r$ be projection onto the last $r$ coordinates. Then $\pi(C)=\Delta_r$, where
  \[ \Delta_r:=\{\ul{x}=(x_1,\ldots,x_r)\in \R_{\geq 0}^r : \|\ul{x}\|_1 \leq 1\}, \]
  where $\|\ul{x}\|_1 := x_1+\cdots+x_r$. The fiber of $C$ above $\ul{x}\in\Delta_r$ along $\pi$ is
  \[ C_{\ul{x}} := (1-\|\ul{x}\|_1 )Q+x_1K_1+\cdots+x_rK_r. \]
  Hence
  \[ \vol_{q+r}(C)=\int_{\Delta_r}\vol_q(C_{\ul{x}})\,\mathrm{d}\ul{x}. \]

  By Theorem and Definition 5.1.7 in Schneider \cite[Eq. (5.28)]{schneider93}, we have
  \begin{align*}
   \vol_q(C_{\ul{x}}) &=\sum_{m=0}^q \sum_{\alpha_1+\cdots+\alpha_r=m} \frac{q!}{(q-m)!\alpha_1!\cdots\alpha_r!} \\
   &\hspace{8em} \times V\big(Q[q-m],K_1[\alpha_1],\ldots,K_r[\alpha_r]\big)\, (1-\|\ul{x}\|_1)^{q-m}x_1^{\alpha_1}\cdots x_r^{\alpha_r},
  \end{align*}
  where $V(\cdots)$ is the mixed volume function, $Q[q-m]$ means $Q$ appears $q-m$ times, and $K_i[\alpha_i]$ means $K_i$ appears $\alpha_i$ times. Since every $K_i$ lies in the fixed $d$-dimensional subspace $\R^d\times \{0\}^{n-d}$, the mixed volume term vanishes whenever $m>d$. Also $K_i\subseteq Q$, so monotonicity of mixed volumes\footnote{This follows by the support function representation of mixed volume (equivalently, by multilinearity together with positivity).} gives
  \[ V\big(Q[q-m],K_1[\alpha_1],\ldots,K_r[\alpha_r]\big)\leq \vol_q(Q). \]
  Therefore
  \begin{align*}
   \vol_{q+r}(C) &\leq \vol_q(Q)\sum_{m=0}^d \sum_{\alpha_1+\cdots+\alpha_r=m}\frac{q!}{(q-m)!\alpha_1!\cdots\alpha_r!} \int_{\Delta_r}(1-\|\ul{x}\|_1)^{q-m}x_1^{\alpha_1}\cdots x_r^{\alpha_r}\,\mathrm{d}\ul{x}.
  \end{align*}

  The Dirichlet integral gives
  \[ \int_{\Delta_r}(1-\|\ul{x}\|_1)^{q-m}x_1^{\alpha_1}\cdots x_r^{\alpha_r}\,\mathrm{d}\ul{x}=\frac{(q-m)!\alpha_1!\cdots\alpha_r!}{(q+r)!}. \]
  Hence
  \[ (q+r)!\vol_{q+r}(C)\leq q!\vol_q(Q) \sum_{m=0}^d \#\{\ul{\alpha}\in \Z_{\geq 0}^r : \|\ul{\alpha}\|_1 = m\}. \]
  Since, by stars-and-bars, $\#\{\ul{\alpha}\in \Z_{\geq 0}^r : \|\ul{\alpha}\|_1 = m\} = \binom{m+r-1}{r-1}$, we obtain
  \[ \NVol(C)\le \NVol(Q)\sum_{m=0}^d\binom{m+r-1}{r-1}=\binom{d+r}{r}\NVol(Q). \qedhere \]
 \end{proof}
 
 %%%%%%%%%%%%%%%%%%%%%%%%%%%%%%%%%%%%%%%%%%%%%%%%%%%%%%%%%%%%%%%%%%
 The following corollary is an immediate iterated application of Lemma \ref{mixedvol}.
 
 \begin{cor}\label{mixedvol-iter}
  Let $n\geq 1$, and let $Q\subseteq \R^n$ be a $q$-dimensional convex body for some $q\leq n$. Let $J\geq 0$, and for each $0\leq j\leq J$, let $r_j\geq 0$ be an integer and let
  \[ K_{j,1},\ldots,K_{j,r_j}\subseteq Q\cap L_j \]
  be convex bodies, where $L_j\subseteq \R^n$ is a linear subspace of dimension at most $d_j$.

  Let $\ul{f}_{j,1},\ldots,\ul{f}_{j,r_j}$ be linearly independent vectors, linearly independent also from $\R^n$, and define
  \[ C:=\conv\Big((Q,\ul{0}),\ (K_{j,\ell},\ul{f}_{j,\ell}) : 0\leq j\leq J,\ 1\leq \ell\leq r_j\Big)\subseteq \R^n\times \R^{r_0+\cdots+r_J}. \]
  Then
  \[ \NVol(C)\leq \NVol(Q)\prod_{j=0}^J \binom{d_j+r_j}{r_j}. \]
 \end{cor}
 
 %%%%%%%%%%%%%%%%%%%%%%%%%%%%%%%%%%%%%%%%%%%%%%%%%%%%%%%%%%%%%%%%%%
 We now prove Proposition \ref{ublbPN}. 
 
 \begin{proof}[Proof of Proposition \ref{ublbPN}]
  Recall the definitions of $\eps$, $q$, $k$, $m_i^{-}$, $m_i^{+}$, $\mathcal{P}_i$, $L_N$ and $U_N$.
  
  \medskip
  \noindent
  $\bullet~$\underline{Lower bound.}
  Since $1\in \mathcal S(N,\sqrt{N})$, and $p^2\in \mathcal{S}(N,\sqrt{N})$ for every prime $p\leq \sqrt{N}$, we have
  \[ 2\Delta(\sqrt{N})\subseteq P(N,\sqrt{N})\subseteq P_N. \]
  Now let $p\in \mathcal{P}_i$. By definition, $p\leq \sqrt{N}/q^i$. Since $m_i^{-} = \lfloor q^i\sqrt{N} \rfloor$, it follows that
  \[ pm_i^{-}\leq \frac{\sqrt{N}}{q^i}\, q^i\sqrt{N}\leq N. \]
  This means $\ul{e}_p + \ul{e}_{\ell} \in P_N$ for every prime $\ell \leq m_i^{-}$, so $\ul{e}_p+\Delta(m_i^-)\subseteq P_N$. Hence
  \[ L_N\subseteq P_N \]
  which by Lemma \ref{Nvol-comp}, implies
  \begin{equation}
   \log\NVol(L_N)\leq \log V_N. \label{lbVN}
  \end{equation}

  \medskip
  \noindent
  $\bullet~$\underline{Upper bound -- main range.}
  For the upper bound, by Lemma \ref{PNdecomp},
  \[ P_N = \conv\bigg(P(N,\sqrt{N})\cup \bigcup_{\sqrt{N}<p\leq N} (\ul{e}_p+P_{N/p})\bigg) = \conv(\mathrm{main} \cup \mathrm{tail}), \]
  where
  \[ \mathrm{main} := P(N,\sqrt{N})\cup \bigcup_{i=1}^{k} \bigcup_{p\in \mathcal{P}_i} (\ul{e}_p+P_{N/p}), \qquad \mathrm{tail} := \bigcup_{p>\sqrt{N}/\eps} (\ul{e}_p+P_{N/p}). \]

  By Lemma \ref{outsimp}, we have $P(N,\sqrt{N})\subseteq \Sigma(N,\sqrt{N})$ and $T(\Sigma(N,\sqrt{N}))=2\Delta(\sqrt{N})$, where $T:\R^{\pi(N)}\to \R^{\pi(N)}$ is the diagonal map
  \[ T(u_p):=\frac{\log p}{\log \sqrt{N}}\,u_p \qquad (p\leq \sqrt{N}), \]
  acting as the identity on the coordinates $p>\sqrt{N}$. Now let $p\in \mathcal{P}_i$. Since $N/p\leq m_i^{+}$, every point $\ul{u}\in P_{N/p}$ has support on the prime coordinates $\ell\leq N/p\leq m_i^{+}$, and satisfies
  \[ \sum_{\ell\leq m_i^{+}} u_\ell \log \ell \leq \log(N/p)\leq \log m_i^{+}. \]
  Therefore $T(\ul{u})$ also has support on the coordinates $\ell\leq m_i^{+}$, and
  \[ \sum_{\ell\leq m_i^{+}} T(u_{\ell}) = \frac{1}{\log \sqrt{N}} \sum_{\ell\leq m_i^{+}} u_{\ell} \log \ell \leq \frac{\log m_i^{+}}{\log \sqrt{N}} \leq 1. \]
  Hence
  \[ T(P_{N/p})\subseteq \frac{\log m_i^{+}}{\log \sqrt{N}} \Delta(m_i^{+})\subseteq \Delta(m_i^{+}). \]
  It follows that the transformed main range is contained in
  \begin{equation}
   T(\mathrm{main}) \subseteq U_N = \conv\bigg(2\Delta(\sqrt{N})\cup \bigcup_{i=1}^k\bigcup_{p\in \mathcal{P}_i} (\ul{e}_p+\Delta(m_i^{+}))\bigg). \label{Tmain}
  \end{equation}
  Moreover,
  \begin{equation}
   \log \det(T)^{-1} = \sum_{p\leq \sqrt{N}} \log\bigg(\frac{\log \sqrt{N}}{\log p}\bigg) \ll \frac{\sqrt{N}}{(\log N)^2}. \label{detT}
  \end{equation}
  
  \medskip
  \noindent
  $\bullet~$\underline{Upper bound -- tail.}
  Now define
  \[ \mathcal{T}_N(\eps;U_N):=\conv\bigg(U_N\cup \bigcup_{p>\sqrt{N}/\eps} (\ul{e}_p+T(P_{N/p}))\bigg). \]
  To estimate the contribution of the tail, put $x_j:=2^{-j}\eps$ for $j\geq 0$, and define
  \[ \mathcal{Q}_j := \bigg\{p\in \P:\frac{\sqrt{N}}{x_j}<p\leq \min \bigg(N,\frac{\sqrt{N}}{x_{j+1}}\bigg)\bigg\}. \]
  Since $\mathcal{Q}_j = \varnothing$ once $\sqrt{N}/x_j>N$, only finitely many of these sets are nonempty.

  Let $p\in \mathcal{Q}_j$. Then $N/p < x_j\sqrt{N}$. Hence every point of $P_{N/p}$ is supported on the prime coordinates $\ell\leq x_j\sqrt{N}$ and satisfies
  \[ \sum_{\ell\leq x_j\sqrt{N}} u_\ell \log \ell \leq \log(N/p) \leq \log(x_j\sqrt{N}). \]
  Therefore $T(P_{N/p})$ is supported on the prime coordinates $\ell\leq x_j\sqrt{N}$ and, for every $\ul{u}\in P_{N/p}$,
  \[ \sum_{\ell\leq x_j\sqrt{N}} T(u_\ell) = \frac{1}{\log\sqrt{N}}\sum_{\ell\leq x_j\sqrt{N}} u_\ell \log \ell \leq \frac{\log(x_j\sqrt{N})}{\log\sqrt{N}} \leq 1. \]
  It follows that
  \[ T(P_{N/p})\subseteq \Delta(\lfloor x_j\sqrt{N}\rfloor). \]

  Set $d_j:=\pi(\lfloor x_j\sqrt{N}\rfloor)$. Then $\Delta(\lfloor x_j\sqrt{N}\rfloor)$ is contained in the coordinate subspace spanned by $\{\ul{e}_\ell:\ell\leq x_j\sqrt{N}\}$, which has dimension $d_j$, and also
  \[ \Delta(\lfloor x_j\sqrt{N}\rfloor)\subseteq \Delta(\sqrt{N})\subseteq U_N. \]
  Applying Corollary \ref{mixedvol-iter} with base body $Q=U_N$, and with the $j$th block consisting of the bodies $T(P_{N/p})$ for $p\in \mathcal{Q}_j$, we obtain
  \[ \log \NVol(\mathcal{T}_N(\eps;U_N)) \leq \log \NVol(U_N)+\sum_{j\geq 0}\log \binom{d_j+|\mathcal{Q}_j|}{|\mathcal{Q}_j|}. \]

  It remains to estimate the sum. We split the indices $j$ with $|\mathcal{Q}_j|\neq 0$ into two classes:
  \[ \mathcal J_1:=\{j : x_j\sqrt{N}\geq N^{1/4}\}, \qquad \mathcal J_2:=\{j : x_j\sqrt{N}<N^{1/4}\}. \]
  If $j\in \mathcal J_1$, then $\log(x_j\sqrt{N})\geq \frac{1}{4}\log N$, and the prime number theorem gives
  \[ d_j=\pi(\lfloor x_j\sqrt{N}\rfloor)\ll x_j \frac{\sqrt{N}}{\log N}, \]
  and also $|\mathcal{Q}_j|\ll \frac{1}{x_j} \sqrt{N}/\log N$.
  Hence
  \[ \log \binom{d_j+|\mathcal{Q}_j|}{|\mathcal{Q}_j|}\ll d_j\log\bigg(1+\frac{|\mathcal{Q}_j|}{d_j}\bigg)\ll x_j \frac{\sqrt{N}}{\log N} \log\bigg(1+\frac{1}{x_j^2}\bigg). \]
  Summing over $j\in \mathcal J_1$, we get
  \begin{align*}
   \sum_{j\in \mathcal J_1}\log \binom{d_j+|\mathcal{Q}_j|}{|\mathcal{Q}_j|} &\ll \frac{\sqrt{N}}{\log N}\sum_{j\geq 0} x_j\log\bigg(1+\frac{1}{x_j^2}\bigg) \\
   &\leq \frac{\sqrt{N}}{\log N}\sum_{j\geq 0} 2^{-j}\eps\bigg(\log 2 + \log\bigg(\frac{1}{\eps^2}\bigg)+2j\log 2 \bigg) \\
   &\ll \eps\log\bigg(\frac{1}{\eps}\bigg) \frac{\sqrt{N}}{\log N}.
  \end{align*}

  Now let $j\in \mathcal J_2$. Then $x_j\sqrt{N}<N^{1/4}$, so
  \[ d_j\leq \pi(N^{1/4})\ll N^{1/4}. \]
  Also trivially $|\mathcal{Q}_j|\leq \pi(N)\ll N/\log N$. Therefore
  \[ \log \binom{d_j+|\mathcal{Q}_j|}{|\mathcal{Q}_j|}\ll d_j\log(1+|\mathcal{Q}_j|)\ll N^{1/4}\log N. \]
  Since there are only $O(\log N)$ indices $j$ with $|\mathcal{Q}_j|\neq 0$, it follows that
  \[ \sum_{j\in \mathcal J_2}\log \binom{d_j+|\mathcal{Q}_j|}{|\mathcal{Q}_j|}\ll N^{1/4}(\log N)^2. \]

  Combining the two ranges, we conclude that
  \[ \log \NVol(\mathcal{T}_N(\eps;U_N)) \leq \log \NVol(U_N)+ O\bigg(\eps\log\bigg(\frac{1}{\eps}\bigg)\frac{\sqrt{N}}{\log N}\bigg) + O(N^{1/4}(\log N)^2). \]
  Since, by \eqref{detT}, $T$ changes normalized volume by at most $O(\sqrt{N}/(\log N)^2)$, it follows from \eqref{Tmain} that
  \begin{equation}
   \log V_N\leq \log \NVol(U_N) + O\bigg(\eps\log\bigg(\frac{1}{\eps}\bigg)\frac{\sqrt{N}}{\log N}\bigg) + O\bigg(\frac{\sqrt{N}}{(\log N)^2}\bigg). \label{ubVN}
  \end{equation}
  Combining this with the lower bound \eqref{lbVN} completes the proof.
 \end{proof}
 
%%%%%%%%%%%%%%%%%%%%%%%%%%%%%%%%%%%%%%%%%%%%%%%%%%%%%%%%%%%%%%%%%%
\subsection{Asymptotics for \texorpdfstring{$L_N$}{L\_N} and \texorpdfstring{$U_N$}{U\_N}}
 The convex bodies $L_N$ and $U_N$ split naturally into blocks: large primes are grouped by $\mathcal{P}_i$, while small primes are grouped according to the nested simplices $\Delta(m_i^\pm)$. Since the coordinates within each block play identical roles, we collapse them to the sum of their coordinates. This allows us to express the volume as an integral over a lower dimensional region.
 
%%%%%%%%%%%%%%%%%%%%%%%%%%%%%%%%%%%%%%%%%%%%%%%%%%%%%%%%%%%%%%%%%%
 \begin{lem}\label{model-fiber}
  Let $d_0\geq d_1\geq \cdots \geq d_k\geq 0$ and $r_1,\ldots,r_k\geq 1$. Put
  \[ \mathcal{U} := \conv\bigg(2\Delta(d_0)\cup \bigcup_{i=1}^k\bigcup_{j=1}^{r_i} (\ul{e}_{i,j} + \Delta(d_i))\bigg), \]
  where the $\ul{e}_{i,j}$ are new independent directions. Set
  \[ a_i:=d_i-d_{i+1}\quad (0\leq i<k), \qquad a_k:=d_k, \qquad M:=d_0+\sum_{i=1}^k r_i. \]
  Then
  \[ \NVol(\mathcal{U}) = \frac{M!}{\prod_{i:\,a_i>0}(a_i-1)!\prod_{i=1}^k(r_i-1)!} \int_{\mathcal{R}_k} \prod_{i:\,a_i>0}u_i^{a_i-1}\prod_{i=1}^k v_i^{r_i-1}\,\mathrm{d}\ul{u}\,\mathrm{d}\ul{v}, \]
  where, writing $\ul{u} = (u_0,\ldots,u_k)$, $\ul{v}=(v_1,\ldots,v_k)$,
  \[ \mathcal{R}_k:=\bigg\{(\ul{u},\ul{v})\in \R_{\geq 0}^{k+1}\times \R_{\geq 0}^k : \|\ul{v}\|_1\leq 1,\ u_0+\cdots+u_i \leq 2(1-\|\ul{v}\|_1)+v_1+\cdots+v_i\ (i\leq k)\bigg\}. \]
 \end{lem}
 \begin{proof}
  For $0\leq i<k$, define $I_i:=\{d_{i+1}+1,\ldots,d_i\}$, and $I_k:=\{1,\ldots,d_k\}$.
  Then $|I_i|=a_i$, and the sets $I_0,\ldots,I_k$ form a partition of $\{1,\ldots,d_0\}$. For $\ul{x}\in \R_{\geq 0}^{d_0}$, define
  \[ U_i(\ul{x}):=\sum_{m\in I_i}x_m \qquad (0\leq i\leq k). \]
  
  \medskip
  \noindent
  $\bullet~$\underline{Claim.}
  Given $t_0,\ldots,t_k\geq 0$, we have $\ul{x}\in t_0\Delta(d_0)+\cdots+t_k\Delta(d_k)$ if and only if
  \begin{equation}
   U_0(\ul{x})+\cdots+U_i(\ul{x})\leq t_0+\cdots+t_i \qquad (0\leq i\leq k). \label{Uoft}
  \end{equation}

  To prove the necessity, write $\ul{x}=\ul{x}^{(0)}+\cdots+\ul{x}^{(k)}$ with $\ul{x}^{(i)}\in t_i\Delta(d_i)$. Since $\ul{x}^{(j)}$ is supported on the first $d_j$ coordinates, it has no coordinates in $I_0\cup \cdots \cup I_i$ when $j>i$. Hence
  \begin{align*}
   U_0(\ul{x})+\cdots+U_i(\ul{x}) = \sum_{m=1}^{d_i}x_m = \sum_{j=0}^i \sum_{m=1}^{d_i}x_m^{(j)} \leq \sum_{j=0}^i t_j,
  \end{align*}
  since each $\ul{x}^{(j)}\in t_j\Delta(d_j)$ has $\ell^1$-norm at most $t_j$.

  To prove the sufficiency, suppose that $\ul{x}\in \R_{\geq 0}^{d_0}$ satisfies \eqref{Uoft}. We will construct vectors $\ul{x}^{(0)}, \ldots, \ul{x}^{(k)}$ with $\ul{x}^{(i)}\in t_i\Delta(d_i)$ for $0\leq i\leq k$, and $\ul{x}=\ul{x}^{(0)}+\cdots+\ul{x}^{(k)}$. We choose numbers $b_{j,i}\geq 0$ for $0\leq j\leq i\leq k$ such that
  \[ \sum_{j=0}^i b_{j,i} = U_i(\ul{x}) \quad (0\leq i\leq k), \quad \text{and}\quad \sum_{i=j}^k b_{j,i}\leq t_j \quad (0\leq j\leq k). \]
  We do this inductively on $i$.

  For $i=0$, the inequality $U_0(\ul{x})\leq t_0$ allows us to set $b_{0,0}:=U_0(\ul{x})$. Suppose now that the $b_{j,m}$ have been chosen for all $m<i$. For $0\leq j\leq i$, define
  \[ c_j:=t_j-\sum_{m=j}^{i-1} b_{j,m}. \]
  Then
  \begin{align*}
   \sum_{j=0}^i c_j = \sum_{j=0}^i t_j-\sum_{m=0}^{i-1}\sum_{j=0}^m b_{j,m} = \sum_{j=0}^i t_j-\sum_{m=0}^{i-1}U_m(\ul{x}) \geq U_i(\ul{x}),
  \end{align*}
  by \eqref{Uoft} for $i$. Hence we can choose nonnegative numbers $b_{0,i},\ldots,b_{i,i}$ such that
  \[ b_{j,i}\leq c_j \quad (0\leq j\leq i), \qquad b_{0,i}+\cdots+b_{i,i}=U_i(\ul{x}), \]
  completing the induction.

  We now define $\ul{x}^{(j)}$ as follows. For each $i$ with $U_i(\ul{x})>0$, set $x_m^{(j)}:=b_{j,i}x_m/U_i(\ul{x})$ for $m\in I_i$ and $0\leq j\leq i$, and set $x_m^{(j)}:=0$ if $U_i(\ul{x})=0$ or $j>i$. Then $\ul{x}^{(j)}$ is supported on the first $d_j$ coordinates and
  \[ \sum_{m\in I_i}x_m^{(j)}=b_{j,i}\qquad (j\leq i\leq k), \qquad \|\ul{x}^{(j)}\|_1=\sum_{i=j}^k b_{j,i}\leq t_j. \]
  Hence $\ul{x}^{(j)}\in t_j\Delta(d_j)$. Moreover, for every $m\in I_i$, we have $\sum_{j=0}^k x_m^{(j)}=x_m$, since $\sum_{j=0}^i b_{j,i}=U_i(\ul{x})$. Thus $\ul{x}=\ul{x}^{(0)}+\cdots+\ul{x}^{(k)}$, proving the claim.

  \medskip
  \noindent
  $\bullet~$\underline{Describing $\mathcal{U}$.}
  Write a point of the ambient space as
  \[ (\ul{x},\ul{y}^{(1)},\ldots,\ul{y}^{(k)})\in \R^{d_0}\times \R^{r_1}\times \cdots \times \R^{r_k}, \]
  where $\ul{y}^{(i)}=(y_{i,1},\ldots,y_{i,r_i})$. Define
  \[ u_i:=U_i(\ul{x}) \quad (0\leq i\leq k), \qquad v_i:=y_{i,1}+\cdots+y_{i,r_i} \quad (1\leq i\leq k). \]

  We first prove that if $(\ul{x},\ul{y}^{(1)},\ldots,\ul{y}^{(k)})\in \mathcal{U}$, then $(\ul{u},\ul{v})\in \mathcal{R}_k$. If $(\ul{x},\ul{y}^{(1)},\ldots,\ul{y}^{(k)})\in \mathcal{U}$, then there exist numbers $\lambda_{i,j}\geq 0$ such that $\sum_{i=1}^k\sum_{j=1}^{r_i}\lambda_{i,j}\leq 1$,
  \[ \ul{y}^{(i)}=(\lambda_{i,1},\ldots,\lambda_{i,r_i}) \qquad (1\leq i\leq k), \]
  and
  \[ \ul{x}\in \bigg(1-\sum_{i=1}^k\sum_{j=1}^{r_i}\lambda_{i,j}\bigg)2\Delta(d_0)+\sum_{i=1}^k \bigg(\sum_{j=1}^{r_i}\lambda_{i,j}\bigg)\Delta(d_i). \]
  Since
  \[ v_i=\sum_{j=1}^{r_i}\lambda_{i,j}, \qquad \|\ul{v}\|_1=\sum_{i=1}^k\sum_{j=1}^{r_i}\lambda_{i,j}\leq 1, \]
  the previous claim, applied with
  \[ t_0:=2(1-\|\ul{v}\|_1), \qquad t_i:=v_i \quad (1\leq i\leq k), \]
  yields, by \eqref{Uoft},
  \[ u_0+\cdots+u_i\leq 2(1-\|\ul{v}\|_1)+v_1+\cdots+v_i \qquad (0\leq i\leq k). \]
  Thus $(\ul{u},\ul{v})\in \mathcal{R}_k$.

  %%%%%%%%%%%%%%%%%%%%%%%%%%%%%%%%%%%%%%%%%%%%%%%%%%%%%%%
  Conversely, let $(\ul{u},\ul{v})\in \mathcal{R}_k$, let $\ul{x}\in \R_{\geq 0}^{d_0}$ satisfy $U_i(\ul{x})=u_i$ for $0\leq i\leq k$, and, for each $1\leq i\leq k$, let $\ul{y}^{(i)}\in \R_{\geq 0}^{r_i}$ satisfy $\|\ul{y}^{(i)}\|_1=v_i$. By the previous claim, the inequalities defining $\mathcal{R}_k$ imply that
  \[ \ul{x}\in (1-\|\ul{v}\|_1)\,2\Delta(d_0)+v_1\Delta(d_1)+\cdots+v_k\Delta(d_k). \]
  Hence there exist vectors $\ul{z}^{(0)}\in 2\Delta(d_0)$, $\ul{z}^{(i)}\in \Delta(d_i)$ for $1\leq i\leq k$, such that
  \[ \ul{x}=(1-\|\ul{v}\|_1)\ul{z}^{(0)}+v_1\ul{z}^{(1)}+\cdots+v_k\ul{z}^{(k)}. \]
  Since $\sum_{j=1}^{r_i} y_{i,j} = v_i$ for $1\leq i\leq k$, we may rewrite this as
  \[ \ul{x}=(1-\|\ul{v}\|_1)\ul{z}^{(0)}+\sum_{i=1}^k\sum_{j=1}^{r_i} y_{i,j}\,\ul{z}^{(i)}. \]
  Therefore 
  \[ (\ul{x},\ul{y}^{(1)},\ldots,\ul{y}^{(k)}) = (1-\|\ul{v}\|_1)(\ul{z}^{(0)},\ul{0},\ldots,\ul{0}) + \sum_{i=1}^k\sum_{j=1}^{r_i} y_{i,j}\,(\ul{z}^{(i)},\ul{0},\ldots,\ul{0},\ul{e}_{j},\ul{0},\ldots,\ul{0}), \]
  where $\ul{e}_{j}$ has a $1$ in the $j$-th coordinate and sits on the $i$-th block $\R^{r_i}$. The coefficients in this convex combination are nonnegative and sum to
  \[ (1-\|\ul{v}\|_1)+\sum_{i=1}^k\sum_{j=1}^{r_i} y_{i,j} = (1-\|\ul{v}\|_1)+\sum_{i=1}^k v_i = 1. \]
  Moreover, $(\ul{z}^{(0)},\ul{0},\ldots,\ul{0})\in 2\Delta(d_0)$, and
  $(\ul{z}^{(i)},\ul{0},\ldots,\ul{0},\ul{e}_{j},\ul{0},\ldots,\ul{0})\in \ul{e}_{i,j}+\Delta(d_i)$ for $1\leq i\leq k$, $1\leq j\leq r_i$. Hence $(\ul{x},\ul{y}^{(1)},\ldots,\ul{y}^{(k)})\in \mathcal U$. %
  %%%%%%%%%%%%%%%%%%%%%%%%%%%%%%%%%%%%%%%%%%%%%%%%%%%%%%%
  We have thus shown that the map
  \[ \Pi:\mathcal{U}\to \R_{\geq 0}^{k+1}\times \R_{\geq 0}^k, \qquad \Pi(\ul{x},\ul{y}^{(1)},\ldots,\ul{y}^{(k)})=(\ul{u},\ul{v}), \]
  has image exactly $\mathcal{R}_k$.

  \medskip
  \noindent
  $\bullet~$\underline{Conclusion.}
  Fix $(\ul{u},\ul{v})\in \mathcal{R}_k$. The fiber $\Pi^{-1}(\ul{u},\ul{v})$ is the direct product of the sets
  \[ \{\ul{z}\in \R_{\geq 0}^{a_i}:\|\ul{z}\|_1=u_i\} \quad (0\leq i\leq k,\ a_i>0),\qquad \{\ul{w}\in \R_{\geq 0}^{r_i}:\|\ul{w}\|_1=v_i\} \quad (1\leq i\leq k). \]
  Each of these is a simplex, with Euclidean volume
  \[ \frac{\sqrt{a_i}\,u_i^{a_i-1}}{(a_i-1)!}\quad (0\leq i\leq k,\ a_i>0), \qquad \frac{\sqrt{r_i}\,v_i^{r_i-1}}{(r_i-1)!} \quad (1\leq i\leq k). \]
  On the other hand, since $\Pi$ takes the sum of the coordinates in each of these mutually orthogonal blocks, its Jacobian is $\sqrt{\det(D\Pi D\Pi^{T})} = \prod_{i:a_i>0}\sqrt{a_i}\prod_{i=1}^k\sqrt{r_i}$. Calculating the volume of $\mathcal{U}$ by slicing into level sets, these factors cancel, and the contribution of the fiber above $(\ul{u},\ul{v})$ is
  \[ \prod_{i:a_i>0}\frac{u_i^{a_i-1}}{(a_i-1)!}\prod_{i=1}^k \frac{v_i^{r_i-1}}{(r_i-1)!}. \]
  Hence,\footnote{If $a_i=0$, then $I_i=\varnothing$ and hence $u_i=0$ identically; so such variables are omitted from $\mathrm{d}\ul{u}$.} 
  \[ \vol(\mathcal{U}) = \int_{\mathcal{R}_k} \prod_{i:a_i>0}\frac{u_i^{a_i-1}}{(a_i-1)!}\prod_{i=1}^k \frac{v_i^{r_i-1}}{(r_i-1)!}\,\mathrm{d}\ul{u}\,\mathrm{d}\ul{v}. \]
  Multiplying by $M!$ yields the formula for $\NVol(\mathcal{U})$.
 \end{proof}

 We now estimate the volume of the model body $\mathcal{U}$ from Lemma \ref{model-fiber}, in the regime where the dimensions of the blocks grow proportionally to a common scale $A_N\to\infty$.
 
%%%%%%%%%%%%%%%%%%%%%%%%%%%%%%%%%%%%%%%%%%%%%%%%%%%%%%%%%%%%%%%%%%
 \begin{lem}\label{model-asymp}
  Let $A_N\to \infty$. Fix $k\geq 1$, and assume
  \[ a_i(N)=\alpha_i A_N+o(A_N)\quad (0\leq i\leq k), \qquad r_i(N)=\beta_i A_N+o(A_N)\quad (1\leq i\leq k), \]
  with $\alpha_i\geq 0$ and $\beta_i>0$ fixed. Let $\mathcal{U}_N$ be the body from Lemma \ref{model-fiber}, built from the data $a_i(N),r_i(N)$. Then
  \[ \frac{1}{A_N}\log \NVol(\mathcal{U}_N)\to \Phi_k(\boldsymbol{\alpha},\boldsymbol{\beta}), \]
  where
  \[ \Phi_k(\boldsymbol{\alpha},\boldsymbol{\beta}) := S\log S-\sum_{i=0}^k \alpha_i\log \alpha_i-\sum_{i=1}^k \beta_i\log \beta_i+\sup_{(\ul{u},\ul{v})\in \mathcal{R}_k}\bigg(\sum_{i=0}^k \alpha_i\log u_i+\sum_{i=1}^k \beta_i\log v_i\bigg), \]
  and $S:=\sum_{i=0}^k \alpha_i+\sum_{i=1}^k \beta_i$.
 \end{lem}
 \begin{proof}
  By Lemma \ref{model-fiber},
  \begin{align*}
   &\NVol(\mathcal{U}_N) \\
   &\hspace{+3em} = \frac{M_N!}{\prod_{i:\,a_i(N)>0}(a_i(N)-1)!\prod_{i=1}^k(r_i(N)-1)!}\int_{\mathcal{R}_k}\prod_{i:\,a_i(N)>0}u_i^{a_i(N)-1}\prod_{i=1}^k v_i^{r_i(N)-1}\,\mathrm{d}\ul{u}\,\mathrm{d}\ul{v},
  \end{align*}
  where $M_N:=d_0(N)+\sum_{i=1}^k r_i(N)=\sum_{i=0}^k a_i(N)+\sum_{i=1}^k r_i(N)$. By Stirling's formula,
  \begin{align*}
   &\frac{1}{A_N}\log\bigg(\frac{M_N!}{\prod_{i:\,a_i(N)>0}(a_i(N)-1)!\prod_{i=1}^k(r_i(N)-1)!}\bigg)\\
   &\hspace{+12em} \to S\log S-\sum_{i=0}^k \alpha_i\log \alpha_i-\sum_{i=1}^k \beta_i\log \beta_i \qquad \text{as } N\to\infty.
  \end{align*}

  For the integral, define
  \[ F_N(\ul{u},\ul{v}) := \sum_{i:\,a_i(N)>0}\frac{a_i(N)-1}{A_N}\log u_i+\sum_{i=1}^k \frac{r_i(N)-1}{A_N}\log v_i, \]
  so that
  \[ \int_{\mathcal{R}_k}\prod_{i:\,a_i(N)>0}u_i^{a_i(N)-1}\prod_{i=1}^k v_i^{r_i(N)-1}\,\mathrm{d}\ul{u}\,\mathrm{d}\ul{v} = \int_{\mathcal{R}_k} e^{A_NF_N(\ul{u},\ul{v})}\,\mathrm{d}\ul{u}\,\mathrm{d}\ul{v}. \]
  Since $k$ is fixed, the region $\mathcal{R}_k$ is a fixed compact polytope. On the relative interior of $\mathcal{R}_k$, the functions $F_N$ converge pointwise to
  \[ F(\ul{u},\ul{v})=\sum_{i=0}^k \alpha_i\log u_i+\sum_{i=1}^k \beta_i\log v_i, \]
  where terms with zero coefficient are interpreted as zero. Moreover, if any variable with positive coefficient tends to $0$, then $F(\ul{u},\ul{v})\to -\infty$. The upper bound
  \[ \limsup_{N\to\infty}\frac{1}{A_N}\log \int_{\mathcal{R}_k} e^{A_NF_N(\ul{u},\ul{v})}\,\mathrm{d}\ul{u}\,\mathrm{d}\ul{v}\leq \sup_{\mathcal{R}_k} F \]
  follows from compactness. For the matching lower bound, choose a point $(\ul{u}_0,\ul{v}_0)$ in the relative interior of $\mathcal{R}_k$ such that
  \[ F(\ul{u}_0,\ul{v}_0)\geq \sup_{\mathcal{R}_k}F-\eta. \]
  By continuity, there is a neighbourhood $W\subseteq \mathcal{R}_k$ of $(\ul{u}_0,\ul{v}_0)$ on which all coordinates are bounded away from $0$. Since $F_N\to F$ uniformly on $W$, there is an integer $N_0$ such that for all $N\geq N_0$ and all $(\ul{u},\ul{v})\in W$,
  \[ F_N(\ul{u},\ul{v})\geq \sup_{\mathcal{R}_k}F-2\eta. \]
  Hence, for all $N\geq N_0$,
  \[ \int_{\mathcal{R}_k} e^{A_NF_N(\ul{u},\ul{v})}\,\mathrm{d}\ul{u}\,\mathrm{d}\ul{v} \geq \mathrm{vol}(W)\, e^{A_N(\sup_{\mathcal{R}_k}F-2\eta)}. \]
  Since $\eta>0$ is arbitrary, this gives
  \[ \frac{1}{A_N}\log \int_{\mathcal{R}_k} e^{A_NF_N(\ul{u},\ul{v})}\,\mathrm{d}\ul{u}\,\mathrm{d}\ul{v} \to \sup_{\mathcal{R}_k} F. \]

  Combining the asymptotics of the factorial term and the integral proves the lemma.
 \end{proof}

 The optimization problem appearing in Lemma \ref{model-asymp} can be simplified by passing to its Lagrange dual, reducing the supremum over $\mathcal{R}_k$ to an infimum over ordered variables.
 
%%%%%%%%%%%%%%%%%%%%%%%%%%%%%%%%%%%%%%%%%%%%%%%%%%%%%%%%%%%%%%%%%%
 \begin{lem}\label{dual-formula}
  Let $\Phi_k$ be as in Lemma \ref{model-asymp}. For all coefficients $\boldsymbol{\alpha} = (\alpha_0,\ldots,\alpha_k)\in \R_{\geq 0}^{k+1}$ and $\boldsymbol{\beta} = (\beta_1,\ldots,\beta_k)\in \R_{>0}^k$,
  \[ \Phi_k(\boldsymbol{\alpha}, \boldsymbol{\beta})=\inf_{0\leq s_1\leq \cdots \leq s_k\leq 1}\bigg(S\log 2-\sum_{i=1}^k \alpha_i\log(1-s_i)-\sum_{i=1}^k \beta_i\log(1+s_i)\bigg). \]
 \end{lem}
 \begin{proof}
  Set
  \[ F(\ul{u},\ul{v}) := \sum_{i=0}^k \alpha_i\log u_i+\sum_{i=1}^k \beta_i\log v_i. \]
  We dualize the maximization of $F$ over $\mathcal{R}_k$.

  For $0\leq i\leq k$, write
  \[ U_i := u_0+\cdots+u_i, \qquad V_i := v_1+\cdots+v_i, \qquad V := v_1+\cdots+v_k. \]
  The defining inequalities of $\mathcal{R}_k$ are
  \[ U_i-2(1-V)-V_i\leq 0 \quad (0\leq i\leq k), \qquad V-1\leq 0. \]
  Introduce nonnegative Lagrange multipliers $\lambda_0,\ldots,\lambda_k,\mu$. The Lagrangian is
  \[ \mathcal{L}(\ul{u},\ul{v};\lambda,\mu) := \sum_{i=0}^k \alpha_i\log u_i+\sum_{i=1}^k \beta_i\log v_i -\sum_{i=0}^k \lambda_i(U_i-2(1-V)-V_i)-\mu(V-1). \]
  Define
  \[ L_i:=\lambda_i+\cdots+\lambda_k \qquad (0\leq i\leq k). \]
  Then $L_0\geq L_1\geq \cdots \geq L_k\geq 0$. Also, the coefficient of $v_i$ in the Lagrangian is $2L_0-L_i+\mu$.
  Hence
  \[ \mathcal{L}(\ul{u},\ul{v};\lambda,\mu) = \sum_{i=0}^k (\alpha_i\log u_i- L_i u_i) +\sum_{i=1}^k (\beta_i\log v_i-(2L_0-L_i+\mu)v_i) +2L_0+\mu. \]
  Maximizing with respect to the variables $u_i$ and $v_i$ gives\footnote{Using that $\sup_{t>0} (a\log t - bt) = a\log a - a - a\log b$.}
  \begin{align*}
   \sup_{\ul{u},\ul{v}}\mathcal{L}(\ul{u},\ul{v};\lambda,\mu) &= \sum_{i=0}^k (\alpha_i\log \alpha_i-\alpha_i-\alpha_i\log L_i) \\
   &\hspace{+5em} + \sum_{i=1}^k (\beta_i\log \beta_i-\beta_i-\beta_i\log(2L_0-L_i+\mu)) + 2L_0 + \mu.
  \end{align*}
  Therefore\footnote{Here we use strong Lagrange duality $\sup_{(\ul{u},\ul{v})\in \mathcal{R}_k} F(\ul{u},\ul{v}) = \inf_{\lambda_i,\mu\geq 0} \sup_{\ul{u},\ul{v}} \mathcal{L}(\ul{u},\ul{v};\lambda,\mu)$, which applies since $F$ is concave, the constraints defining $\mathcal{R}_k$ are affine, and Slater's condition holds (see \cite[\S5.2.3]{boydvand04}).}
  \begin{align*}
   \Phi_k(\boldsymbol{\alpha}, \boldsymbol{\beta}) &= S\log S-S + \inf_{\lambda_i,\mu\geq 0} \bigg(2L_0+\mu-\sum_{i=0}^k \alpha_i\log L_i -\sum_{i=1}^k \beta_i\log(2L_0-L_i+\mu) \bigg).
  \end{align*}

  If we scale all multipliers by a factor $t>0$, then every $L_i$ and every $2L_0-L_i+\mu$ is also scaled by $t$. Hence the quantity inside the infimum becomes
  \[ t(2L_0+\mu)-S\log t+\text{(term independent of $t$)}. \]
  For fixed multipliers, this is minimized when
  \[ t=\frac{S}{2L_0+\mu}. \]
  Hence, at the minimum, we may normalize so that $2L_0+\mu=S$. Set
  \[ a_i:=\frac{L_i}{S} \qquad (0\leq i\leq k). \]
  Then $0\leq a_k\leq \cdots \leq a_1\leq a_0\leq \frac{1}{2}$, and
  \[ \frac{2L_0-L_i+\mu}{S}=1-a_i. \]
  Therefore
  \[ \Phi_k(\boldsymbol{\alpha}, \boldsymbol{\beta})=\inf_{0\leq a_k\leq \cdots \leq a_1\leq a_0\leq \frac{1}{2}} \bigg(-\sum_{i=0}^k \alpha_i\log a_i-\sum_{i=1}^k \beta_i\log(1-a_i)\bigg). \]
  For fixed $a_1,\ldots,a_k$, the function is decreasing in $a_0$, so the minimum occurs at $a_0=\frac{1}{2}$. Now put $s_i:=1-2a_i$ for $1\leq i\leq k$. Then $0\leq s_1\leq \cdots \leq s_k\leq 1$, and $a_i=\frac{1-s_i}{2}$, $1-a_i=\frac{1+s_i}{2}$. Substituting this into the previous formula gives
  \[ \Phi_k(\boldsymbol{\alpha},\boldsymbol{\beta}) = \inf_{0\leq s_1\leq \cdots \leq s_k\leq 1} \bigg(S\log 2-\sum_{i=1}^k \alpha_i\log(1-s_i)-\sum_{i=1}^k \beta_i\log(1+s_i)\bigg). \qedhere \]
 \end{proof}

 Under a natural monotonicity condition on the coefficients, the optimization problem in Lemma \ref{dual-formula} can be solved explicitly.
 
%%%%%%%%%%%%%%%%%%%%%%%%%%%%%%%%%%%%%%%%%%%%%%%%%%%%%%%%%%%%%%%%%%
 \begin{lem}\label{geom-eval}
  Let $\Phi_k$ be as in Lemma \ref{model-asymp}. Suppose $\alpha_i>0$ and $\beta_i>0$ for $1\leq i\leq k$, the ratios $\beta_i/\alpha_i$ are strictly increasing in $i$, and $\beta_i\geq \alpha_i$ for all $1\leq i\leq k$. Then
  \[ \Phi_k(\boldsymbol{\alpha},\boldsymbol{\beta}) = \alpha_0\log 2 + \sum_{i=1}^k \Big((\alpha_i+\beta_i)\log(\alpha_i+\beta_i)-\alpha_i\log \alpha_i-\beta_i\log \beta_i\Big). \]
 \end{lem}
 \begin{proof}
  By Lemma \ref{dual-formula}, we must minimize
  \[ S\log 2-\sum_{i=1}^k \alpha_i\log(1-s_i)-\sum_{i=1}^k \beta_i\log(1+s_i) \]
  subject to $0\leq s_1\leq \cdots \leq s_k\leq 1$.
  Ignoring the monotonicity constraints, the derivative with respect to $s_i$ is
  \[ \frac{\alpha_i}{1-s_i}-\frac{\beta_i}{1+s_i}. \]
  Hence the unique critical point is
  \[ s_i=\frac{\beta_i-\alpha_i}{\alpha_i+\beta_i}. \]
  Since $\beta_i\geq \alpha_i$, these values lie in $[0,1]$. Because the ratios $\beta_i/\alpha_i$ are strictly increasing, the values $s_i$ are nondecreasing in $i$. Thus they satisfy the constraints and yield the constrained minimum. Substituting gives
  \[ 1-s_i=\frac{2\alpha_i}{\alpha_i+\beta_i}, \qquad 1+s_i=\frac{2\beta_i}{\alpha_i+\beta_i}, \]
  and therefore
  \begin{align*}
   (\alpha_i+\beta_i)\log 2-\alpha_i\log(1-s_i)\,-\,&\beta_i\log(1+s_i) \\
   &= (\alpha_i+\beta_i)\log(\alpha_i+\beta_i)-\alpha_i\log \alpha_i-\beta_i\log \beta_i.
  \end{align*}
  Summing over $i=1,\ldots,k$ leaves the extra term $\alpha_0\log 2$, and the formula follows.
 \end{proof}
 
 By continuity, we extend $\Phi_k$ to nonnegative coefficients $\alpha_i,\beta_i$, using the same formula as in Lemma \ref{model-asymp} and the convention $0\log 0:=0$. The dual formula of Lemma \ref{dual-formula} extends to these coefficients in the same way.

 %%%%%%%%%%%%%%%%%%%%%%%%%%%%%%%%%%%%%%%%%%%%%%%%%%%%%%%%%%%%%%%%%%
 \begin{lem}\label{Phi-perturb}
  Let $\Phi_k$ be as in Lemma \ref{model-asymp}, and suppose
  \[ \sum_{i=0}^k \alpha_i+\sum_{i=1}^k \beta_i\leq B. \]
  Let $0\leq j\leq k$ and $1\leq \ell\leq k$, and define
  \[ \alpha_i' := \begin{cases}
                   \alpha_i, & i\neq j,\\
                   0, & i=j,
                  \end{cases} \qquad
     \beta_i' := \begin{cases}
                  \beta_i, & i\neq \ell,\\
                  0, & i=\ell.
                 \end{cases} \]
  Then
   \[ 0\leq \Phi_k(\boldsymbol{\alpha},\boldsymbol{\beta})-\Phi_k(\boldsymbol{\alpha}',\boldsymbol{\beta}') \ll \alpha_j\bigg(1+\log\frac{B}{\alpha_j}\bigg) + \beta_\ell\bigg(1+\log\frac{B}{\beta_\ell}\bigg), \]
 \end{lem}
 \begin{proof}
  We first estimate the effect of removing $\alpha_j$. For $0\leq s\leq \alpha_j$, define
  \[ \Psi(s):=\Phi_k(\alpha_0,\ldots,\alpha_{j-1},s,\alpha_{j+1},\ldots,\alpha_k;\ \boldsymbol{\beta}). \]
  Let
  \[ S_s:=s+\sum_{\substack{0\leq i\leq k\\ i\neq j}}\alpha_i+\sum_{i=1}^k\beta_i. \]
  Fix $s\in (0,\alpha_j]$, and let $(\ul{u},\ul{v})\in \mathcal R_k$ be a maximizing point for $\Psi(s)$. Then
  \[ \Psi(s)-\Psi(0) \leq (S_s\log S_s-s\log s+s\log u_j)-S_0\log S_0 = \int_0^s \log\bigg(\frac{S_tu_j}{t}\bigg)\,\mathrm{d}t. \]
  Since $(\ul{u},\ul{v})\in \mathcal R_k$, we have $u_j\leq u_0+\cdots+u_k\leq 2$, and since $S_t\leq B$, it follows that
   \[ \Psi(s)-\Psi(0)\leq \int_0^s \log\bigg(\frac{2B}{t}\bigg)\,\mathrm{d}t = s\log\bigg(\frac{2eB}{s}\bigg) \ll s\bigg(1+\log\frac{B}{s}\bigg). \]
  Taking $s=\alpha_j$, we obtain
  \[ \Phi_k(\boldsymbol{\alpha},\boldsymbol{\beta})-\Phi_k(\boldsymbol{\alpha}',\boldsymbol{\beta})\ll \alpha_j\bigg(1+\log\frac{B}{\alpha_j}\bigg). \]
  The same argument applies to the removal of $\beta_\ell$. Hence,
  \begin{align*}
   \Phi_k(\boldsymbol{\alpha},\boldsymbol{\beta})-\Phi_k(\boldsymbol{\alpha}',\boldsymbol{\beta}') &= (\Phi_k(\boldsymbol{\alpha},\boldsymbol{\beta})-\Phi_k(\boldsymbol{\alpha}',\boldsymbol{\beta})) + (\Phi_k(\boldsymbol{\alpha}',\boldsymbol{\beta})-\Phi_k(\boldsymbol{\alpha}',\boldsymbol{\beta}')) \\
   &\ll \alpha_j\bigg(1+\log\frac{B}{\alpha_j}\bigg) + \beta_\ell\bigg(1+\log\frac{B}{\beta_\ell}\bigg).
  \end{align*}

  It remains only to note that the left-hand side is nonnegative. This follows from Lemma \ref{dual-formula}: in the dual formula, the coefficient $\alpha_j$ appears multiplied by $\log 2-\log(1-s_j)\geq 0$ if $j\geq 1$, while $\alpha_0$ appears multiplied by $\log 2$, and the coefficient $\beta_\ell$ appears multiplied by $\log 2-\log(1+s_\ell)\geq 0$. Thus decreasing either coefficient cannot increase the infimum.
 \end{proof}
 
 %%%%%%%%%%%%%%%%%%%%%%%%%%%%%%%%%%%%%%%%%%%%%%%%%%%%%%%%%%%%%%%%%%
 We are now ready to prove Theorem \ref{thm2pi}.
 
 \begin{proof}[Proof of Theorem \ref{thm2pi}]
  Let $L_N$ and $U_N$ be as in \eqref{Lbody} and \eqref{Ubody}. By Proposition \ref{ublbPN}, for every fixed $0<\eps<1$ and every $q\in(0,1)$, we have
  \[ \log \NVol(L_N)\leq \log V_N\leq \log \NVol(U_N) + O\bigg(\eps\log\bigg(\frac{1}{\eps}\bigg)\frac{\sqrt N}{\log N}\bigg) + O\bigg(\frac{\sqrt N}{(\log N)^2}\bigg). \]
  We now identify $L_N$ and $U_N$ with the model bodies of Lemma \ref{model-asymp}. Since $k$ is fixed once $\eps$ and $q$ are fixed, all asymptotics below are uniform in $0\leq i\leq k$. Set
  \[ d_0^{\pm}(N) := \pi(\sqrt N), \quad d_i^{\pm}(N):=\pi(m_i^{\pm}) \quad (1\leq i\leq k), \qquad r_i(N):=|\mathcal P_i| \quad (1\leq i\leq k). \]
  Then $L_N$ (resp. $U_N$) is exactly the body from Lemma \ref{model-asymp} built from the data $d_i^{-}(N)$ (resp. $d_i^{+}(N)$) and $r_i(N)$. Note that for every fixed $c>0$, the prime number theorem gives
  \[ \pi(c\sqrt N) = (2c+o(1))\frac{\sqrt N}{\log N}. \]
  Since $m_i^{-}=q^i\sqrt N+O(1)$ and $m_i^{+}=q^{i-1}\sqrt N+O(1)$, it follows that
  \[ \begin{cases}
      d_i^{-}(N)=(2q^i+o(1))\dfrac{\sqrt N}{\log N} &(0\leq i\leq k), \\
      d_i^{+}(N)=(2q^{i-1}+o(1))\dfrac{\sqrt N}{\log N} &(1\leq i\leq k), \qquad d_0^{+}(N)=(2+o(1))\dfrac{\sqrt N}{\log N}.
     \end{cases} \]
  The corresponding quantities $a_i(N)$ in that lemma are
  \[ a_i^{\pm}(N):=d_i^{\pm}(N)-d_{i+1}^{\pm}(N) \quad (0\leq i<k), \qquad a_k^{\pm}(N):=d_k^{\pm}(N). \]
  Likewise,
  \[ |\mathcal P_i| = \pi\bigg(\frac{\sqrt N}{q^i}\bigg)-\pi\bigg(\frac{\sqrt N}{q^{i-1}}\bigg) = (2(1-q)q^{-i}+o(1)) \frac{\sqrt N}{\log N} \qquad (1\leq i\leq k).
  \]
  Thus Lemma \ref{model-asymp} applies with
  \[ \alpha_i^{-} = \begin{cases}
                     2(1-q)q^i,& 0\leq i<k,\\
                     2q^k,& i=k,
                    \end{cases} \qquad
     \alpha_i^{+} = \begin{cases}
                     0,& i=0,\\
                     2(1-q)q^{i-1},& 1\leq i<k,\\
                     2q^{k-1},& i=k,
                    \end{cases} \]
  and $\beta_i=2(1-q)q^{-i}$ for $1\leq i\leq k$, and we obtain
  \begin{equation}
   \frac{\log N}{\sqrt N} \log \NVol(L_N)\to \Phi_k(\boldsymbol{\alpha}^{-},\boldsymbol{\beta}), \qquad \frac{\log N}{\sqrt N}\log \NVol(U_N)\to \Phi_k(\boldsymbol{\alpha}^{+},\boldsymbol{\beta}). \label{limvol}
  \end{equation}
  
  Using Lemma \ref{Phi-perturb}, since $q^k\asymp \eps$, we may remove the pair $(\alpha_k^\pm,\beta_k)$ at a total cost of
  \[ O\bigg(\eps\log\bigg(\frac{1}{\eps}\bigg)\bigg) + o_{q\to 1,\eps}(1). \]
  After removing this pair, we identify the resulting quantity with $\Phi_{k-1}$ for the remaining coefficients. We have
  \[ \frac{\beta_i}{\alpha_i^-}=q^{-2i}, \qquad \frac{\beta_i}{\alpha_i^+}=q^{1-2i} \qquad (1\leq i\leq k-1), \]
  and also $\beta_i\geq \alpha_i^\pm$. Thus Lemma \ref{geom-eval} applies, and we obtain
  \begin{equation*}
   \begin{gathered}
    \Phi_k(\boldsymbol{\alpha}^-,\boldsymbol{\beta})=\alpha_0^-\log 2+\sum_{i=1}^{k-1}2(1-q)q^i\,\phi(q^{-2i})+O\bigg(\eps\log\bigg(\frac{1}{\eps}\bigg)\bigg)+o_{q\to 1,\eps}(1), \\
    \Phi_k(\boldsymbol{\alpha}^+,\boldsymbol{\beta})=\sum_{i=1}^{k-1}2(1-q)q^{i-1}\,\phi(q^{1-2i})+O\bigg(\eps\log\bigg(\frac{1}{\eps}\bigg)\bigg)+o_{q\to 1,\eps}(1), 
   \end{gathered}
  \end{equation*}
  where $\phi(t):=(1+t)\log(1+t)-t\log t$.\footnote{Note that $\phi(t) = Q(t+1)-Q(t)+1$ where $Q(t)=t\log t - t +1$.}
  
  Since $\alpha_0^-=2(1-q)$, we have $\alpha_0^-\log 2\to 0$ as $q\to 1^{-}$. Writing $x_i:=q^i$, so that
  \[ 2(1-q)q^i=2(x_i-x_{i+1}), \qquad q^{-2i}=x_i^{-2}, \]
  we see that the lower sum is a Riemann sum for $\int_{\eps}^1 2\,\phi(x^{-2})\,\mathrm{d}x$.
  Likewise,
  \[ 2(1-q)q^{i-1}=2(x_{i-1}-x_i), \qquad q^{1-2i}=(x_{i-1}x_i)^{-1}, \]
  so the upper sum is a Riemann sum for the same integral. Hence,
  \begin{equation}
   \Phi_k(\boldsymbol{\alpha}^\pm,\boldsymbol{\beta}) = (1+o_{q\to 1^{-}}(1)) \int_\eps^1 2\,\phi(x^{-2})\,\mathrm{d}x + o_{\eps\to 0^{+}}(1). \label{rmnnSum}
  \end{equation}
  By Proposition \ref{ublbPN} together with \eqref{limvol}, it remains to evaluate \eqref{rmnnSum}.

  Let
  \[ g(x) := 2\phi(x^{-2}) = 2\bigg(1+\frac{1}{x^2}\bigg)\log(1+x^2) - 4\log x. \]
  Then $g'(x) = -4\log(1+x^2)/x^3$. Integrating by parts gives
  \begin{align*}
   \int_{0^+}^1 g(x)\,dx
   &=[xg(x)]_{0^+}^1-\int_{0^+}^1 xg'(x)\,\mathrm{d}x \\
   &= 4\log 2+ 4\int_{0^+}^1 \frac{\log(1+x^2)}{x^2}\,\mathrm{d}x \\
   &= 4\log 2+ 4\bigg(-\log 2+\int_{0^+}^1 \frac{2\,\mathrm{d}x}{1+x^2}\bigg)= 2\pi.
  \end{align*}
  Thus, letting $q\to 1^{-}$ and $\eps\to 0^{+}$, we conclude that
  \[ \frac{\log V_N}{\sqrt{N}/\log N}\to 2\pi.\qedhere \]
 \end{proof}

%%%%%%%%%%%%%%%%%%%%%%
\addtocontents{toc}{\protect\setcounter{tocdepth}{0}}
\section*{Acknowledgements}
 The authors would like to thank Andrew Granville, Leo Goldmakher and Dimitris Koukoulopoulos for their advice and encouragement.
 
\addtocontents{toc}{\protect\setcounter{tocdepth}{1}}
%%%%%%%%%%%%%%%%%%%%%%

% ----------------------------------------------------------------
\bibliographystyle{amsplain}
\bibliography{refs.bib}%
\end{document}